\def\ZZ         {{\mathbb Z}}
\def\RR         {{\mathbb R}}
\def\CC         {{\mathbb C}}
\def\FF         {{\mathbb F}}
\def\QQ         {{\mathbb Q}}
\def\PP         {{\mathbb P}}
\def\NN         {{\mathbb N}}
\def\ZZ         {{\mathbb Z}}

\def\A         {{\cal A}}
\def\B         {{\cal B}}
\def\C         {{\cal C}}
\def\D           {{\cal D}}

\def\H         {{\cal H}}
\def\I         {{\cal I}}

\def\O          {{\cal O}}

\def\S           {{\cal S}}

\def\V           {{\cal V}}

\def\rk         {{\rm rk}}

\def\log        {{rm log}}

\def\Char       {{\rm Char}}
\def\Ker        {{\rm {Ker}}}
\def\Spec       {{\rm Spec}}
\def\Supp        {{\rm Supp}}

\def\dim        {{\rm dim}}
\def\log        {{\rm log}}

\def\cal        {\mathcal}

\documentclass{amsart}
\usepackage{amssymb}
\usepackage{verbatim}
\usepackage{eucal}
\newtheorem{theorem}{Theorem}[section]

\newtheorem{prop}[theorem]{Proposition}
\newtheorem{corollary}[theorem]{Corollary}

\newtheorem{dfn}[theorem]{Definition}

\newtheorem{cond}[theorem]{Condition}

\theoremstyle{remark}
\newtheorem{remark}[theorem]{Remark}
\newtheorem{claim}[theorem]{Claim}
\newtheorem{example}[theorem]{Example}

\title{Albanese varieties of abelian covers}

\author{Anatoly Libgober}

\dedicatory{To the memory of Shreeram Abhyankar.}

\thanks{Author supported by a grant from Simons Foundation}
\begin{document}
\begin{abstract} We show that the Albanese variety 
of an abelian cover of the projective plane is isogenous 
to a product of isogeny components of abelian varieties
associated with singularities of the ramification locus
provided certain conditions are met.
In particular Albanese varieties 
of abelian covers of $\PP^2$ ramified over arrangements 
of lines and uniformized by the unit ball in $\CC^2$
are isogenous to a product of Jacobians of Fermat curves.
Periodicity of the sequence of (semi-abelian)   
Albanese varieties of unramified cyclic covers of  
complements to a plane singular curve is shown.
\end{abstract}

\maketitle

\section{Introduction}

Albanese varieties of cyclic branched 
covers of $\PP^2$ ramified over singular curves 
are rather special.
If singularities of the ramification locus 
are no worse than ordinary nodes and cusps then 
(cf. \cite{jose})  
the Albanese variety of a cyclic cover is 
isogenous to a product of elliptic curves $E_0$ with $j$-invariant zero.
More generally, in \cite{annalen} it was shown that the Albanese variety  
of a cyclic cover with ramification locus having {\it arbitrary} singularities
is isogenous to 
a product of isogeny components of local Albanese varieties i.e.
the abelian varieties 
canonically associated with the local singularities of the ramification locus. 
In particular, Albanese varieties of cyclic covers 
are isogenous to a product of Jacobians of curves.

In this paper we shall 
describe Albanese varieties of {\it abelian} covers of $\PP^2$.
The main result is that the class of abelian varieties 
which are Albanese varieties of ramified abelian covers 
(with possible non reduced ramification locus)
is also built from the isogeny components of local  Albanese varieties, 
provided some conditions on fundamental group of the complement 
to ramification locus are met (cf. \ref{condition}).
Also, 
in abelian case one needs to allow local Albanese varieties of non reduced 
singularities having the same reduced structure as the 
germs of the singularities 
of ramification locus of the abelian cover.

One of the steps in our proof of this result involves a 
description of Jacobians 
of abelian covers of projective line having an independent interest.
 In this case we show that
 all isogeny components of Jacobians of abelian covers 
of $\PP^1$ with arbitrary ramification 
are 
components of Jacobians of explicitly described 
cyclic covers. If the abelian 
cover is ramified only at three points and has the  
Galois group isomorphic to $\ZZ_n^2$ then it  
is biholomorphic to Fermat curve $x^n+y^n=z^n$.
In this case, such results are
going  back to works of Gross, Rohrlich and Coleman
(cf. \cite{gross},\cite{coleman})
where isogeny components of  Jacobians of Fermat curves were studied.

The proof of isogeny decomposition 
of abelian covers is constructive and,
 as an application, we obtain the isogeny classes of Albanese 
varieties of the abelian covers of $\PP^2$, 
discovered by Hirzebruch (cf.\cite{hirz}),
having the unit ball as the universal cover. 
These Albanese varieties 
are isogenous to products of Jacobians of Fermat curves
described explicitly.
Another
interesting abelian cover of $\PP^2$ ramified over an arrangement of lines 
is the Fano surface of lines on the Fermat cubic threefold.
The Albanese variety of this Fano surface (according 
to \cite{clemens}, this abelian variety is also 
the intermediate Jacobian of the Fermat cubic 
threefold) is isogenous to the product of five copies of $E_0$.
This result was recently independently obtained in \cite{roulleau} and 
\cite{toledo} (in \cite{roulleau} the {\it isomorphism} class of Albanese 
variety of Fano surfaces was found).

Another application considers the behavior of the Albanese varieties 
in the towers of cyclic and abelian covers. It is known 
for some time that Betti and Hodge numbers of cyclic (resp. abelian)
covers are periodic (resp. polynomially periodic cf. \cite{eko}).
It turns out that the sequence of isogeny classes of 
Albanese varieties of 
cyclic covers with given ramification locus is periodic
but periodicity fails in abelian towers. Moreover, 
we show similar periodicity for sequence of semi-abelian 
varieties which are Albanese varieties of {\it quasi-projective 
surfaces} which are unramified covers of $\PP^2\setminus \C$.

The content of the paper is the following. In section \ref{prelim}
 we recall several key definitions and results 
used later, in particular, the characteristic varieties,
Albanese varieties in quasi-projective and local cases.
Section \ref{linejacobians} considers Jacobians of 
abelian covers of $\PP^1$, and the main result is 
that isogeny components of such Jacobians 
are all the isogeny components of Jacobians of cyclic covers of $\PP^1$.
This section also contains calculation 
of multiplicities of characters of representation of the covering 
group on the space of holomorphic 1-forms. In the case of 
cyclic covers, such multiplicities were calculated in 
\cite{Arch}. The main result of the paper, showing that
Albanese varieties of abelian covers are isogenous 
to a product of isogeny components of local Albanese varieties
of singularities, is proven in section  \ref{main}.
The case of covers ramified over arrangements of lines 
is considered in section \ref{examplessection}. This includes, 
the already mentioned case of Fano surface (of lines) on the Fermat cubic 
threefold.
The last section 
contains applications to calculation 
of Mordell-Weil ranks of isotrivial abelian varieties and   
periodicity properties of Albanese varieties in towers
of abelian covers.
Note that the prime field of all varieties,
maps between them and function fields considered in this 
paper is $\CC$.

I want to thank anonymous referee for careful reading 
of this paper and many useful suggestions including
usage of LaTex.

\section{Preliminaries}\label{prelim}

\subsection{Characteristic varieties}

We recall the construction of invariants of the fundamental 
group of the complement playing the key role in description 
of the Albanese varieties of abelian covers. We follow \cite{charvar}
(cf. also \cite{depth}).

Let $X$ be  a quasi-projective smooth manifold such that 
$H_1(X,\ZZ) \ne 0$. The exact
sequence
\begin{equation}\label{keyextension} 
      0 \rightarrow \pi_1(X)'/\pi_1(X)'' 
\rightarrow \pi_1(X)/\pi_1(X)'' \rightarrow \pi_1(X)/\pi_1(X)' 
\rightarrow 0
\end{equation}
(where $G'$ denotes the commutator subgroup of a group $G$) 
can be used to define the action of $H_1(X,\ZZ)=
\pi_1(X)/\pi_1(X)'$ on the left term in (\ref{keyextension}).
This action allows to view $C(X)=\pi_1(X)'/\pi_1(X)'' \otimes \CC$ as
a $\CC[H_1(X,\ZZ)]$-module. Recall that the support of 
a module $M$ over a commutative 
noetherian ring $R$ is the sub-variety $\Supp(M) \subset \Spec(R)$ consisting 
of the prime ideals $\wp$ for which the localization $M_{\wp} \ne 0$. 
\begin{dfn}\label{charvar} The characteristic variety $V_i(X)$ 
is ({\it the  reduced}) sub-variety of $\Spec \CC[H_1(X)]$ 
which is the support $\Supp(\Lambda^i(C(X)))$  of the $i$-th exterior 
power of the module $C(X)$. 
The depth of $\chi \in \Spec \CC[H_1(X)]$ is an integer given by  
\begin{equation}
d(\chi)=\{ {\rm max}\  i 
\vert \chi \in V_i(X)\}
\end{equation}
\end{dfn}
Using the canonical identification of $\Spec\CC[H_1(X,\ZZ)]$ and 
the torus of characters $\Char(\pi_1(X))$
 one can interpret points of characteristic varieties
as rank one local systems on $X$. This interpretation  
leads to the following alternative description of $V_i(X)$
(cf. \cite{ekogrenoble}, 
\cite{charvar}))  
\begin{equation}
  V_i(X) \setminus {1}=\{\chi \in \Char\pi_1(X) \vert , \chi \ne 1,
dim H^1(X,\chi) \ge i\}
\end{equation}

It follows from \cite{arapura} that if a smooth projective 
closure  $\bar X$ of $X$ satisfies\footnote{this condition 
is independent of a choice of smooth compactification  $\bar X$}
$H_1(\bar X,\QQ)=0$
then each $V_i(X)$ is a finite union 
of translated subgroups of the affine torus $\Char(\pi_1(X))$ 
i.e. a finite union of subset 
of the form $\psi \cdot H$ where $H$ is a subgroup 
of $\Char(\pi_1(X))$ and $\psi$ is a character of $\pi_1(X)$.
Moreover, such a character $\psi$ can be chosen 
to have a finite order (cf. \cite{nonvan}).  
It also follows from \cite{arapura}
that each irreducible 
component  $\V$ of characteristic variety having a dimension greater than one
determines a holomorphic map:
$\nu: X \rightarrow P$ where $P$ is a hyperbolic curve (i.e. a curve
with negative euler characteristic).

In the case when $X=\PP^2\setminus \C$, where $\C$ is a plane curve
with arbitrary singularities, 
$P$ is biholomorphic to $\PP^1\setminus D$ 
where $D$ is a finite set.

Returning to the case when $X$ is smooth quasi-projective, 
a component corresponding to a map $\nu: X \rightarrow P$ 
consists of the characters $\nu^*(\chi)$ where $\chi \in \Char(\pi_1(P))$ 
(here, for a map $\phi: X \rightarrow Y$ between topological spaces $X,Y$, 
we denote by $\phi^*$ the induced map 
$\Char (H_1(Y,\ZZ))=H^1(Y,\CC^*) \rightarrow H^1(X,\CC^*)=\Char (H_1(X,\ZZ))$).
The map $\nu$ also induces homomorphisms  $h^i(\nu*): H^i(P,\chi) \rightarrow 
H^i(X,\nu^*(\chi))$ and
$h_i(\nu*): H_i(P,\chi) \rightarrow 
H_i(X,\nu^*(\chi))$.
The maps $h^1(\nu^*)$ and $h_1(\nu^*)$ are isomorphisms 
for all but finitely many $\chi \in \Char(\pi_1(P))$
(cf.\cite[Proof of Prop.1.7]{arapura}).

At the intersection of components the depth of characters 
is bigger then the depth of generic character in either 
of the components i.e. the depth is jumping. More precisely, 
if $\chi \in V_k(X) \cap V_l(X)$
where both $V_k(X)$ and $V_l(X)$ have positive dimensions then 
the depth of $\chi$ is at least $k+l$ (cf. \cite{matei}).
More precisely we shall use the following  
assumption on the characteristic variety 
at the points belonging to several components. 
In particular it includes an inequality on depth in the {\it the opposite} 
direction:

\begin{cond}\label{condition} 
(1) Let $\chi \in \V_1\cap...\cap \V_s$ 
and $\chi=\nu^*_i(\chi_i)$ for $\chi_i \in \Char(P_{i})$ where 
$\nu_i: X \rightarrow P_i$ is the map corresponding to 
the component $\V_i$. Then:
\begin{equation}
  \bigoplus_i h_1(\nu_i): H_1(X,\chi) \rightarrow \bigoplus H_1(P_i,\chi_i)
\end{equation} 
is injective. In particular, 
the depth of each character $\chi$ in the intersection 
of several positive dimensional irreducible components $\V_1,...,\V_s$ 
of the characteristic variety
does not exceed the sum of the depths of 
the generic character in each component $\V_i$.

(2) If $\chi \in \V_i$ but $\chi \notin \V_i\cap \V_j, j \ne i$
then $h_1(\nu_i): H_1(X,\chi) \rightarrow H_1(P_i,\chi_i)$
is an isomorphism.

\end{cond}

This condition is satisfied in the examples considered in section 
\ref{examplessection}.

\subsection{Abelian covers.}

Given a surjection $\pi_{\Gamma}: \pi_1(X) \rightarrow 
\Gamma$ onto a finite group, there are a unique quasi-projective manifold 
$\widetilde{X}_{\Gamma}$ and a map $\tilde \pi_{\Gamma}: 
\widetilde X_{\Gamma} \rightarrow X$
which is an unramified 
cover with covering group $\Gamma$. The variety $\widetilde X_{\Gamma}$ 
is characterized by the property that $\Gamma$ acts freely on 
$ \widetilde X_{\Gamma}$ and $\widetilde X_{\Gamma}/\Gamma=X$. 
Let $\bar X_{\Gamma}$ denote a smooth model of a compactification of 
$\widetilde{X}_{\Gamma}$ such that $\tilde \pi_{\Gamma}$ 
extends to a regular map $\bar \pi_{\Gamma}: 
\bar X_{\Gamma} \rightarrow \bar X$ ($\bar X$ as above). 
The fundamental group $X_{\Gamma}$, being 
birational invariant, depends only on $X$ and $\pi_{\Gamma}$.

Let $\C=\bar X\setminus X$ be the ``divisor at infinity'' 
and let $\tilde \C \subset \C$ be a divisor on $\bar X$ whose irreducible 
components are components of $\C$. If $\chi \in \Char(\pi_1(X))$ is 
trivial on the components of $\C$ not in $\tilde \C$ then 
$\chi$ is the pullback of a character of $\pi_1(\bar X \setminus 
\tilde \C)$ via the inclusion $X \rightarrow \bar X \setminus \tilde \C$.
We shall denote the corresponding character of 
$\pi_1(\bar X\setminus \C)$ as  $\chi$ as well but 
(since the depth of $\chi$ depends on the underlying space)
corresponding depths will be denoted $d(\chi,\C)$ and 
$d(\chi, \tilde \C)$ respectively.

The homology groups of unramified and ramified covers can be found in terms 
of characteristic varieties as follows (cf. \cite{charvar}).

\begin{theorem}\label{homologycovers}
 1.(cf. \cite{charvar}) With above notations:
 \begin{equation}
\rk H_1(\widetilde X_{\Gamma},\QQ)=\sum_{\chi \in Char \Gamma} d(\pi_{\Gamma}^*(\chi),
\C)
\end{equation}

2.(cf. \cite{sakuma}) 
Let $I(\chi)$ be the collection of components of $\C$ such 
that $\chi(\gamma_{C_i})\ne 1$ ($\gamma_{C_i}$ is a meridian of 
the component $C_i$) and let $\C_{\chi}=\bigcup_{i \in I(\chi )} C_i$.  
Then  
\begin{equation}\label{sakuma}
\rk H_1(\bar X_{\Gamma},\QQ)=\sum_{\chi \in Char \Gamma} d(\pi_{\Gamma}^*(\chi),
\C_{\pi^*_{\Gamma}(\chi}))
\end{equation}
\end{theorem}

The following special case of Theorem \ref{homologycovers} 
will be used in section \ref{linejacobians}.

\begin{corollary}\label{restrictedcharacters}
 Let 
$$\pi_{\Gamma(a_{i_1},....,a_{i_l})}: \pi_1(\PP^1\setminus \{a_{i_1},...,a_{i_l}\})
\rightarrow H_1(\PP^1\setminus \{a_{i_1},...,a_{i_l}\},\ZZ/n\ZZ), 
0 \le i_1,....i_l, \le k$$ 
be the composition of Hurewicz map with the reduction modulo $n$ and 
let $X_n(a_{i_1},....a_{i_l})$ 
be the corresponding ramified abelian cover\footnote{note that this is the universal cover for the covers having an 
abelian $n$-group as the covering group} 
of $\PP^1$ 
with the covering group 
$\Gamma=H_1(\PP^1\setminus \{a_{i_1},...,a_{i_l}\},\ZZ/n\ZZ)$.
Then 
\begin{equation}
H^1(X_n(a_0,...,a_k),\CC)_{\chi}=
\oplus H^1(X_n(a_{i_1},...,a_{i_l},\CC)_{\chi^{r}(a_{i_1,...,i_l})} \ \ \ 3 \le l \le k, 0 \le i_j \le k
\end{equation} 
where the summation is over the 
characters ${\chi^{r}(a_{i_1,...,i_l})}$ which are restricted in the sense 
that they do not take value 1 on a cycle which is 
the boundary of a small disk about 
any point $a_{i_1},...,a_{i_l}$.
\end{corollary}

\subsection{Albanese varieties of quasi-projective 
manifolds}\label{semiabelian}

Let $X$ be a smooth quasi-projective manifold and
let $\bar X$ be a smooth compactification of $X$.
Denote $\bar X \setminus X$  by $\C$ and assume in this section
that $\C$ is  a divisor with normal crossings.
One associates to $X$ a semi-abelian 
variety i.e. an extension:
\begin{equation}\label{defalb}
0 \rightarrow T \rightarrow Alb(X) \rightarrow A \rightarrow 0
\end{equation}
where 
$T$ is a torus and $A$ is an abelian variety (the abelian part of $Alb(X)$)
called the {\it Albanese variety} of $X$.  
Such a semi-abelian variety can be obtained as
$$H^0(\bar X,\Omega^1(\log(\C))^*/H_1(X,\ZZ)$$ where embedding 
$H_1(X,\ZZ) \rightarrow H^0(\bar X,\Omega^1(\log(\C))^*$
is given by 
$\gamma \in H_1(X,\ZZ) \rightarrow (\omega \rightarrow \int_{\gamma}\omega)$
(and polarization of abelian part is 
coming from the Hodge form on $H_1(\bar X,\ZZ)$ given by  
$(\gamma_1, \gamma_2)=\int_{\bar X} \gamma_1^* \wedge \gamma_2^* \wedge 
h^{\dim X-1}$ where $h \in H^2(\bar X,\ZZ)$ is the class of hyperplane section).

One can also view 
$AlbX$ as the semi-abelian part of the 1-motif associated to the 
(level one) mixed Hodge structure
on $H_1(X,\ZZ)$ (cf. \cite{deligne}, section 10.1). 
The abelian part of $Alb(X)$ is the Albanese variety
of a smooth projective compactification of $X$. It clearly 
is independent of a choice 
of the latter.

In this paper we shall consider 
Albanese varieties of abelian covers of quasi-projective
surfaces but note that the 
 Albanese variety of an abelian covers of quasi-projective
manifold of any dimension can be obtained as the 
Albanese variety of the corresponding abelian cover of a surface 
due to the following Lefschetz type result:

\begin{prop} Let $X$ be a quasi-projective manifold and 
$H \cap X$ a generic 2-dimension section by a linear space $H$.
Then $\pi_1(X)=\pi_1(X \cap H)$. 

Let $\Gamma$ be a finite quotient of these groups.
Then the unramified $\Gamma$-covers 
$\tilde X_{\Gamma}$ and $\widetilde {(X \cap H)_{\Gamma}}$,
corresponding to surjections of  $\pi_1(X)$ and $\pi_1(X \cap H)$
onto $\Gamma$,  
have Albanese varieties which are isomorphic 
as semi-abelian varieties.
\end{prop}

\subsection{Local Albanese varieties of plane curve singularities} For details of the material
of this section we refer to \cite{annalen}.
Let $f(x,y)$ be an analytic 
germ of a reduced isolated curve singularity in $\CC^2$.
One associates with it the Milnor fiber $M_f=B \cap f^{-1}(t)$
where $B$ is a small ball in $\CC^2$ centered at the singular point.
The latter supports canonical level one 
limit Mixed Hodge structure on $H^1(M_f,\ZZ)$
(cf. \cite{vancoh}). Again 
one can apply Deligne's construction \cite[10.3.1]{deligne} 
which leads to the following.

\begin{dfn} The local Albanese variety of a germ $f$ is 
the abelian part of the 1-motif of the limit 
Mixed Hodge structure on $H^1(M_f,\ZZ)$. Equivalently, 
this is quotient of $F^0Gr_{-1}^WH_1(M_f\CC)/Im H_1(M_f,\ZZ)$
where $F$ and $W$ are respectively the Hodge and weight filtrations.
The canonical polarization is coming from the form 
induced by the intersection form of $H_1(M_f,\ZZ)$ on 
$Gr^W_{-1}H_1(M_f,\ZZ)$.
\end{dfn}

The local Albanese has a description in terms of the Mixed Hodge 
structure on the cohomology of the link of the surface singularity 
associated to $f$.

\begin{prop}(cf. \cite{annalen}, Prop.3.1)
 Let $f(x,y)$ be a germ 
of a plane curve with Milnor fiber $M_f$ and
\footnote{this assumption is a somewhat weaker 
than the one in \cite{annalen} but the argument works in this
case with no change} 
for which the semi-simple part of monodromy has order $N$.
Let $L_{f,N}$ the the link 
of the corresponding surface singularity 
\begin{equation}\label{surfacecover} 
z^N=f(x,y)
\end{equation}
Then there is the isomorphism of the mixed Hodge structures:
\begin{equation}\label{lemmaisomorphism}
Gr^W_3H^2(L_{f,N})(1)=Gr_1^WH^1(M_f)
\end{equation}
where the mixed Hodge structure on the left is the Tate twist of 
the mixed Hodge structure constructed in 
\cite{durfee} 
 and the one on the right is the mixed Hodge structure 
on vanishing cohomology constructed in \cite{vancoh}.
\end{prop}

Below we shall use Albanese varieties for non reduced germs 
and those can be define using the abelian part of the 1-motif of
mixed Hodge structure $Gr^W_3H^2(L_{f,N})(1)$.

Recall finally that the local Albanese can be described 
in terms of a resolution of the singularity (\ref{surfacecover}).
\begin{theorem}\label{localcontribution}
 (cf. \cite{annalen} Theorem 3.11)
Let $f(x,y)=0$ be a singularity 
let $N$ be the order of the semi-simple part of its monodromy operator.
The local Albanese variety of germ $f(x,y)=0$ 
is isogenous to the product of the Jacobians of the exceptional 
curves of positive genus for a resolution of the singularity
(\ref{surfacecover}). 
\end{theorem}

\begin{example}\label{examplenonreduced} Consider the non-reduced singularity
\begin{equation}\label{nonreducedsing}
f(x,y)=x^{a_1}(x-y)^{a_2}y^{a_3} \ \ \ a_1+a_2+a_3=n
\end{equation} having the ordinary triple point 
as the corresponding reduced germ. In this case, the local Albanese variety 
is isogeneous to the Jacobian of plane curve whose affine portion is given by
\begin{equation}\label{exceptionalpositive}
v^n=u^{a_1}(u-1)^{a_2}
\end{equation}
Indeed, resolution of (\ref{nonreducedsing}) can be achieved by 
a single blow up. The multiplicity of the exceptional curve 
is equal to $n$. It follows from A'Campo's formula that  
the characteristic polynomial of the monodromy is $(t^n-1)(t-1)$ 
and that the order of the monodromy operator acting 
on $Gr^W_1H^1(M_f)$ is equal to $n$.
A resolution of $n$-fold cyclic cover of the surface 
singularity
\begin{equation}\label{cycliccoversurface}
z^n=x^{a_1}(x-y)^{a_2}y^{a_3}
\end{equation} 
can be obtained by resolving cyclic quotient 
singularities of the normalization of the pullback of 
this covering to the blow up of $\CC^2$ resolving $f_{red}(x,y)=0$
(here $f_{red}$ is corresponding reduced polynomial).
This pull-back has as an open subset the surface 
given in $\CC^3$ by equation:
$w^n=u^nv^{a_1}(v-1)^{a_2}$.
Such resolution of surface (\ref{cycliccoversurface}) has 
only one exceptional curve of positive genus and this  
exceptional curve is the $n$-fold cyclic cover of $\PP^1$ ramified at 3 points. 
The monodromies of this $n$-cover around ramification points are 
multiplications by $exp({{2 \pi \sqrt{-1}a_i}\over n}), i=1,2,3$. 
This allows to identify the exceptional curve with curve
(\ref{exceptionalpositive}).
It follows from the Theorem \ref{localcontribution}
that the local Albanese variety of singularity (\ref{nonreducedsing}), 
as was claimed, 
is isogenous to the Jacobian of curve (\ref{exceptionalpositive}).
\end{example}

\section{Jacobians of abelian covers of a line}\label{linejacobians}

The following will be used in the proof of the theorem \ref{mainresult}.

\begin{theorem}\label{abcyccovers}
 Let $X_n$ be the abelian cover 
of $\PP^1$ ramified at $\A=\{ a_0,a_1,...a_k\}  \subset \PP^1$ 
corresponding to the surjection $\pi_1(\PP^1\setminus \A)\rightarrow
H_1(\PP^1\setminus \A,\ZZ_n)$. 
Let $A_i\in \NN,i=0,....,k$ be a collection 
of integers such that 
\begin{equation}\label{conditionona}
 \sum_{i=0}^{i=k} A_i=0 \ (mod \ n), 1 \le A_i <n \ \ \ 
{\rm gcd}(n,A_0,...,A_k)=1
\end{equation}
Denote by $X_{n\vert A_0,....,A_k}$ a smooth 
 model of 
the cyclic cover of $\PP^1$ which affine portion is given by 
\begin{equation}\label{cycliccover} 
  y^n=(x-a_0)^{A_0} \cdot....\cdot (x-a_k)^{A_k}
\end{equation}
(by (\ref{conditionona}) this model is irreducible).
Then the Jacobian of $X_n$ is isogenous 
to the product of the isogeny components of the Jacobians of the 
curves $X_{n \vert A_0,...A_k}$.
\end{theorem}

\begin{remark}\label{fermat} If $k=2$ then the curve $X_n$ is biholomorphic
to Fermat curve $x^n+y^n=z^n$ in $\PP^2$, since as affine model 
of the abelian cover one can take the curve in $\CC^3$ given by 
$x^n=u, \ y^n=1-u$,
and the above theorem 
follows from the calculations in \cite{gross} containing 
explicit formulas for simple isogeny components of the 
Fermat curves.
\end{remark}

\begin{corollary}\label{abelianiscyclic}
 Let $X_{\Gamma}$ be a covering of $\PP^1$ 
with abelian Galois group $\Gamma$
ramified at $a_0,...,a_k \in \PP^1$. Then  
there exist a collection of curves, each being 
a cyclic covers (\ref{cycliccover}) of $\PP^1$, 
such that the Jacobian of 
$X_{\Gamma}$ is isogenous to a product of isogeny components 
of Jacobians of the curves in this collection.
\end{corollary}

\begin{proof} Let $\pi_{\Gamma}: H_1(\PP^1\setminus \bigcup_{i=0}^{i=k} a_i,\ZZ)
\rightarrow \Gamma$ be the surjection corresponding to the 
covering $X_{\Gamma}$, $\delta_i \in 
H_1(\PP^1\setminus \bigcup_{i=0}^{i=k} a_i,\ZZ), i=0,...,k$ 
be the boundary of a small 
disk about $a_i, i=0,...,k$ and 
let $n_i$ be the order of the element 
$\pi_{\Gamma}(\delta_i)\in \Gamma$. Then for $n=lcm(n_0,...,n_k)$ 
one has a surjection $H_1(\PP^1\setminus \bigcup_{i=0}^{i=k} a_i,\ZZ/n\ZZ)
 \rightarrow \Gamma$ and hence a dominant map $X_n \rightarrow X_{\Gamma}$.
In particular the Jacobian of $X_{\Gamma}$ is a quotient of the 
Jacobian of $X_n$ and the claim follows.
\end{proof}

\begin{proof}[Proof of the theorem  \ref{abcyccovers}]
 We shall assume below that one of ramification points, say $a_0$, is the 
point of $\PP^1$ at infinity.

 A projective model of  $X_n$ can be obtained as 
the projective closure in $\PP^{k+1}$ 
(which homogeneous coordinates we shall denote $x,z_1,...,z_k,w$)
of the complete intersection in $\CC^{k+1}$ given by the equations: 
\begin{equation}\label{abelianmodel}
  z_1^n=x-a_1,.....,z_k^n=x-a_k
\end{equation}
The Galois covering $X_n \rightarrow \PP^1$ is given 
by the restriction on this complete intersection 
of the projection of $\PP^{k+1}$ from the subspace $x=w=0$.

For any $(A_0,A_1,...,A_k)$ as above, consider the map  
\begin{equation}
\Phi_{n\vert,A_0,....,A_k}: X_n \rightarrow X_{n \vert A_0,A_1,...A_k}
\end{equation}  
which in the chart $w \ne 0$ is the restriction 
on $X_n$ of the map $\CC^{k+1} \rightarrow \CC^2$ given by: 
\begin{equation}
   \Phi_{A_1,...,A_k}: (z_1,...,z_k,x) \rightarrow (y,x)=(z_1^{A_1}....z_k^{A_k},x)
\end{equation}

The map $\Phi_{n\vert,A_0,....,A_k}$ is the map of the covering spaces of $\PP^1$
corresponding to the surjection of the Galois groups
 $$H_1(\PP^1\setminus \bigcup_{i=0}^{i=k} a_i,\ZZ/n\ZZ) 
\rightarrow \ZZ/n\ZZ$$ 
which is given by 
\begin{equation}
    (i_0,i_1,...,i_k) \rightarrow \sum_j i_jA_j \ \ {\rm mod} \ n
\end{equation}

The maps $\Phi_{n \vert A_0,....,A_k}$ induce the maps of Jacobians:
\begin{equation}\label{mapjacobians}
   \bigoplus_{A_0,...,A_k, 0 \le A_i <n-1} {(\Phi_{n \vert A_0,...,A_k})}_*: Jac(X_n) \rightarrow \bigoplus 
Jac(X_{n \vert A_0,....,A_k})
\end{equation}

We claim that the kernel of a  map (\ref{mapjacobians}) is finite.
This clearly implies the Theorem \ref{abcyccovers}.
Finiteness for the kernel of morphism (\ref{mapjacobians}) will follow from 
surjectivity of the map of cotangent spaces at respective identities
of Jacobians (\ref{mapjacobians}):
\begin{equation}\label{cohomologymap}
     \bigoplus_{A_0,...,A_k}H^{1,0}(X_{n \vert A_0,....,A_k},\CC) \rightarrow H^{1,0}(X_n,\CC)
\end{equation} 

For each $\chi \in \Char \ZZ/n\ZZ$ let $m_{\chi}^{1,0}(n\vert A_0,...,A_k)$
(resp. $m_{\chi}^{0,1}(n\vert A_0,...,A_k)$) denotes the dimension 
of isotypical summand of $H^{1,0}(X_{n \vert A_0,....,A_k},\CC)$ \newline
(resp. $H^{0,1}(X_{n \vert A_0,....,A_k},\CC)$) on which $\ZZ/n\ZZ$ acts
via the character $\chi$. 
Similarly $m^{1,0}_{\Phi_{n \vert A_0,....,A_k}^*(\chi)}(n)$
(resp. $m^{0,1}_{\Phi_{n \vert A_0,....,A_k}^*(\chi)}(n)$) will denote
the dimension of the eigenspace of the pull back 
$\Phi_{n \vert A_0,....,A_k}^*(\chi) \in \Char 
H_1(\PP^1\setminus \A,\ZZ/n\ZZ)$ for the action of the covering 
group of $X_n \rightarrow \PP^1$ on $H^{1,0}(X_n)$ (resp. $H^{0,1}(X_n)$).

It follows from Theorem \ref{homologycovers} (2), that the depth of 
 $\chi$ considered as a character of 
$H_1(\PP^1\setminus \A,\ZZ)$ can be written as:
\begin{equation}
 d(\chi)=m_{\chi}^{0,1}(n\vert A_0,...,A_k)+m_{\chi}^{1,0}(n\vert A_0,...,A_k)=
\end{equation}
$$ m^{0,1}_{\Phi_{n \vert A_0,....,A_k}^*(\chi)}(n)+
m^{1,0}_{\Phi_{n \vert A_0,....,A_k}^*(\chi)}(n)
$$
Moreover, one has inequalities:
\begin{equation}
  m_{\chi}^{0,1}(n\vert A_0,...,A_k) \le m^{0,1}_{\Phi_{n \vert A_0,....,A_k}^*(\chi)}(n)
\end{equation}
$$m_{\chi}^{1,0}(n\vert A_0,...,A_k) \le m^{1,0}_{\Phi_{n \vert A_0,....,A_k}^*(\chi)}(n)$$
\bigskip
Hence, in fact,
\begin{equation}\label{equalitymultuplicities}
m_{\chi}^{0,1}(n\vert A_0,...,A_k) = m^{0,1}_{\Phi_{n \vert A_0,....,A_k}^*(\chi)}(n)
\end{equation}
$$m_{\chi}^{1,0}(n\vert A_0,...,A_k) = m^{1,0}_{\Phi_{n \vert A_0,....,A_k}^*(\chi)}(n)$$

Now let us fix $\chi \in \Char (H_1(\PP^1\setminus \A,\ZZ/n\ZZ))$, i.e.  
a character of the Galois group of the cover $X_n \rightarrow \PP^1$,
such that its value on the cycle 
$\delta_i \in  H_1(\PP^1\setminus \A,\ZZ/n\ZZ)$ 
corresponding to $a_i \in \PP^1, i=0,...,m$ satisfies:
\begin{equation}\label{nontrivialitycondition}
\chi(\delta_i)=
exp\left( {{2 \pi \sqrt{-1} j_i}\over n}\right) \ne 1, (1 \le j_i <n)
\end{equation}
and let $J=gcd (j_0,....,j_k)$. The collection 
of integers $A_i={j_i \over J}$ satisfies condition (\ref{conditionona}).
Denote by $\Gamma_0$ the cyclic group $\chi(H_1(\PP^1\setminus \A,\ZZ))
\subset \CC^*$. Then $\chi$ can be considered as a character 
$\chi' \in \Char(\Gamma_0)$ and then $\chi=\pi^*(\chi')$ where
$\pi$ is projection of the abelian cover  with covering group
$\Gamma$ onto the cyclic cover with the covering group $\Gamma_0$.
It follows from (\ref{equalitymultuplicities}) 
that any isotypical component in $H^{1,0}(X_n,\CC)_{\chi}$ is the image 
of the isotypical component of a cyclic covers and hence the map 
(\ref{cohomologymap}) is surjective which concludes the proof.
\end{proof}

We shall finish this section with an explicit formula
for $dim H^0(X_n,\Omega^1_{X_n})_{\chi}$ i.e. the 
multiplicity of the isotypical component of the covering group 
of abelian cover acting on the space of holomorphic 1-forms.

\begin{prop}\label{abelianeigenspace}
Let the values of a character $\chi \in \Char H_1(\PP^1\setminus \A,\ZZ/n\ZZ)$, 
$\chi\ne 1$,
be given as in (\ref{nontrivialitycondition}). Assume that 
$J=gcd(j_0,...j_k)=1$ and 
let $M=\sum_i (n-j_i)$. Then
\begin{equation}\label{formulaabelianeigenspace}
dim H^{1,0}(X_n)_{\chi}=\left[{{M}\over n}\right]
\end{equation}
\end{prop}

\begin{remark} If $J \ne 1$ then Prop. \ref{abelianeigenspace} 
yields  an expression for the dimension 
of isotypical component corresponding to $\chi \in \Char 
H_1(\PP^1\setminus \A,\ZZ/n\ZZ)$ as well. Indeed, this dimension coincides 
with the dimension of isotypical component for 
$\chi$ considered as the character of $H_1(\PP^1\setminus \A,\ZZ/
({n \over J}\ZZ))$.
\end{remark}

\begin{proof}[Proof of Prop. \ref{abelianeigenspace}]
The equations of the projective closure of the 
complete intersection (\ref{abelianmodel}) are 
\begin{equation}\label{projectiveequation}
z_i^n=(x-a_iw)w^{n-1}, i=1,...,k
\end{equation}
The only singularity of (\ref{projectiveequation}) occurs at $w=0, z_i=0, x=1$. 
Near it (\ref{projectiveequation}) 
is a complete intersection locally given by $z_i^n=w^{n-1}\gamma_i$ where 
$\gamma_i$ is a unit. It has $n^{k-1}$ branches 
(corresponding to the orbits of the action $(z_1,...,z_k)
\rightarrow (\zeta z_1,...,\zeta z_k), \zeta^n=1$)
each equivalent to $z_i=t^{n-1},w=t^n$. Therefore (\ref{projectiveequation}) 
is a ramified cover 
of $\PP^1$ with $k+1$ branching points $a_1,...,a_k,\infty$ 
over which it has 
$n^{k-1}$ preimages with ramification index $n$ at each ramification 
point.

The space  $H^0(\Omega^1_{X_n})$  (for a smooth model of 
(\ref{projectiveequation})) is 
generated by the residues of $k+1$-forms
\begin{equation}\label{meroform}
 {{z_1^{j_1-1}....z_k^{j_k-1}P(x,w)\Omega}\over {\Pi (z_i^n-(x-a_iw)w^{n-1})}}
\  (1 \le j_i) \ \ 
{\rm where} \ \ 
\sum_1^k (j_i-1)+degP+k+2=nk
\end{equation}
(cf. \cite[Theorem 2.9]{griffiths}). Here 
$$\Omega=\sum_i (-1)^{i-1}z_idz_1\wedge ...\widehat{dz_i}..\wedge dz_k \wedge dx \wedge dw+(-1)^{k+1}(
xdz_1\wedge ...\wedge dz_k \wedge dw-
wdz_1\wedge ...\wedge dz_k \wedge dx)$$ 
In the chart $x \ne 0$ such residue (of (\ref{meroform})) is given by: 
\begin{equation}{{z_1^{j_1-1}....z_k^{j_k-1}P(w)dw}\over {(z_1....z_k)^{n-1}}}
\end{equation}
Using (\ref{projectiveequation}), one can reduce powers of $z_i$ 
i.e. we can assume:
\begin{equation} 
1 \le j_i \le n-1
\end{equation}
and a basis of the eigenspace $H^0(\Omega^1_{X_n})_{\chi}$, with $\chi$  
as in  (\ref{nontrivialitycondition}), 
can be obtained by selecting $P(w)=w^s$ 
where $s$ must satisfy: 
\begin{equation}\label{conditionone}
\sum^k_1 (j_i-1)+s+k+2 \le nk \ \ \ \ \ 
\end{equation}
The adjunction condition assuring that the residue of (\ref{meroform})
will be regular on normalization of (\ref{projectiveequation}) is
\begin{equation}\label{conditiontwo}
-\sum_1^k (n-j_i)(n-1)+sn+n-1 \ge 0  \ \ \ \ 
\end{equation}
To count the number of solutions of (\ref{conditionone})
and (\ref{conditiontwo})
i.e. $dim H^0(\Omega^1_{X_n})_{\chi}$ with $\chi$ 
given by (\ref{nontrivialitycondition}),
let $\bar j_i=n-j_i$. Then $1 \le \bar j_i \le n-1$ and 
(\ref{conditionone}),(\ref{conditiontwo}) have form
$\sum_1^k (n-1-\bar j_i) +s +k+2 \le kn$, 
$-(\sum_1^k \bar j_i)(n-1)+sn+n > 0$.
Hence:
\begin{equation}\label{conditionthree}
s+2 \le \sum_1^k \bar j_i < {{(s+1)n} \over {n-1}}=s+1+
{{s+1} \over {n-1}} \ \ \ \  
\end{equation}
Notice that from (\ref{conditionone}) one has $s \le nk-k-2$ i.e. 
${{s+1} \over {n-1}} \le k-
{1 \over {n-1}}$ and hence $\sum_1^k \bar j_i \le k+s$. In particular 
possible values of $\sum_1^k \bar j_i$ are $s+2,....s+k$ and
therefore for given $\bar j_i$, parameter $s$ can take at most $k-1$ values
$\sum \bar j_i-2, ....,\sum \bar j_i-k$. In particular, multiplicities 
of the $\chi$-eigenspaces do not exceed $k-1$.

Let $\sum \bar j_i=M$. Then from (\ref{conditionthree}) one has
$M-1-{{M } \over n}  < s \le M-2$ and 
hence the number of possible values of $s$ is  
$$M-2-\left[M-1-{{M} \over n}\right]
=-1-\left[-{{M} \over n}\right]=
\left[{{M} \over n}\right]$$
as claimed in the Prop. \ref{abelianeigenspace}.
\end{proof}

\begin{remark} One can deduce the theorem \ref{abcyccovers} 
using Prop. \ref{abelianeigenspace}
and 
the following: 
\begin{prop}\label{cycliceigenspace}(\cite{Arch}, Prop. 6.5).
For $x \in \RR$ denote by $\left\langle x \right\rangle=x-[x]$ the fractional part of $x$.
Assume that $gcd(i,n)=1$ and $n$ does not divide either of $A_0,...,A_k$.
Then for the curve (\ref{cycliccover}) the dimension of the 
eigenspace corresponding to the eigenvalue $exp({{2 \pi \sqrt{-1} i}\over n})$
of the automorphism of
$H^{1,0}(X_{n,A_0,...,A_k},\CC)$ induced by the map 
$(x,y)\rightarrow (x,yexp(-{{2 \pi \sqrt{-1}} \over n}))$ 
equals to  
\begin{equation}\label{archinard} 
-\left\langle{{i\sum_0^k A_s} \over n}\right\rangle+\sum_0^k 
\left\langle{{iA_s} \over n}\right\rangle 
\end{equation}
\end{prop}
Indeed, the equality of multiplicities (\ref{equalitymultuplicities})
follows by comparison (\ref{formulaabelianeigenspace}) with 
(\ref{archinard})
since for $i=1$ (\ref{archinard}) yields $-{{\sum A_s} \over n}
+\left[{{\sum A_s} \over n}\right]+\sum {A_s \over n}=
\left[{{\sum A_s} \over n}\right]$
\end{remark}

\begin{remark} Special case of Prop. \ref{cycliceigenspace} appears 
also in \cite{annalen} (cf. lemma 6.1). The multiplicity of the 
latter corresponds to the case $j=n-i$ in Prop. \ref{cycliceigenspace}. 
\end{remark}

\section{Decomposition theorem for abelian covers of 
plane}\label{main}

The main result of this section 
relates the Albanese variety of ramified covers to 
the local Albanese varieties of ramification locus as follows.

\begin{theorem}\label{mainresult}Let $\C$ be a plane algebraic curve 
such that all irreducible components of its
characteristic variety contain the identity of 
$\Char\pi_1(\PP^2 \setminus \C)$. Assume that the 
Condition \ref{condition} is satisfied. 
Let $\pi_{\Gamma}: \pi_1(\PP^2 \setminus \C) \rightarrow \Gamma$ be 
a surjection onto a finite abelian group. Then the Albanese variety 
of the abelian cover $\bar X_{\Gamma}$ ramified over $\C$ and  
associated with $\pi_{\Gamma}$ is isogenous 
to a product of isogeny components of local Albanese varieties 
of possibly non-reduced germs having as reduced singularity a singularity 
of $\C$.

\end{theorem}

\begin{proof} 
To each component of positive dimension of the characteristic variety 
corresponds an isogeny component of Albanese variety of $\bar X_{\Gamma}$ 
as follows. 

Let $Char_j$ 
be an irreducible component of the characteristic variety $V_1(\PP^2 \setminus 
\C)$ of $\C$  (cf. (\ref{charvar})) and let
$\phi_j: \PP^2 \setminus \C$ $\rightarrow \PP^1\setminus D_j$
be the corresponding holomorphic map
where $D_j$ is a finite subset of
 $\PP^1$. The cardinality of $D_j$ is 
equal to $dim(Char_j)+1$ and 
$Char_j=
\phi_j^*(Char \pi_1(\PP^1 \setminus D_j))$.
Denote by  
$\Gamma_j$ the push-out of $\pi_{\Gamma}$. The map $\phi_j$ 
is dominant and yields the surjection 
$(\phi_j)_*: \pi_1(\PP^2 \setminus \C) \rightarrow \pi_1(\PP^1\setminus D_j)$
of the fundamental groups.
With these notations we have the universal (for the groups filings 
 the right left corner of) 
commutative diagram:
\begin{equation}\label{pushout}
\begin{matrix}\pi_1(\PP^2\setminus \C) & \rightarrow & \pi_1(\PP^1\setminus D_j)
 \cr \downarrow &  & \downarrow \cr 
\Gamma & \rightarrow & \Gamma_j \cr
\end{matrix}
\end{equation}
A character of $H_1(\PP^2 \setminus \C,\ZZ)$,
which is the image of a character of $\Gamma$ for the 
map $\Char \Gamma \rightarrow \Char H_1(\PP^2 \setminus \C,\ZZ)$,
can be obtained as a pullback of a character of 
$H_1(\PP^1 \setminus D_j)$ if and only if it is a pullback 
of a character of $\Gamma_j$
via maps in diagram (\ref{pushout}). 
Let $\D_j \rightarrow \PP^1$ the ramified cover with 
branching locus $D_j$, having  
$\Gamma_j$ as its Galois group 
and let $\Phi_j: Alb(\bar X_{\Gamma}) \rightarrow Jac(\D_j)$ 
be the corresponding Albanese map. The Jacobian 
$Jac(\D_j)$ is an isogeny component
of $Alb(\bar X_{\Gamma})$. It depends only on $Char_j$ and $\Gamma$.

Next let $\chi_k, k=1,..,N$ be the collection of characters of 
$\pi_1(\PP^2 \setminus \C)$ whose depth is greater than the depth 
of generic point on the component of characteristic variety 
to which it belongs. We shall call such characters 
{\it the jumping characters of $\C$}. 
It follows form our Condition \ref{condition}
that jumping characters are exactly the intersection points 
of the components of characteristic variety.

We claim injectivity of the map of Albanese varieties induced by 
the holomorphic maps $\phi_j$:
\begin{equation}\label{theoremexactseq}
0 \rightarrow  Alb(\bar X_{\Gamma}) \buildrel {\bigoplus \Phi_j}  
 \over  \rightarrow \bigoplus_j Jac(\D_j) 
\end{equation}

To see that $\Ker \bigoplus \Phi_j=0$,  consider the 
induced homomorphism 
\begin{equation}\label{homomorphismalbaneses}
H_1(Alb(\bar X_{\Gamma}),\CC) \rightarrow 
H_1(\bigoplus_j Jac(\D_j), \CC).
\end{equation}
The group $\Gamma$ acts on 
both vector spaces and the  homomorphism 
(\ref{homomorphismalbaneses})
is $\Gamma$-equivariant.
For a character 
$\chi$ belonging to a single component of characteristic
variety the depths is equal to the depth of 
the generic character in its component (cf. Condition \ref{condition})
which in turn coincides with $H_1(\D_j,\CC)_{\chi}$.
Therefore one has isomorphism $H_1(\bar X_{\Gamma},\CC)_{\chi} \rightarrow 
H_1(\D_j,\CC)_{\chi}$.
For a character $\chi=\chi_k$, i.e. for a character in the intersection of 
several components, again from Condition \ref{condition},
one has injection: 
$H_1(\bar X_{\Gamma},\CC)_{\chi} \rightarrow \oplus_{j,\chi \in \Char_j}
H_1(\D_j,\CC)$. This implies (\ref{theoremexactseq}).

To finish the proof of the Theorem \ref{mainresult} 
it suffices to show that each summand
in the last term in  (\ref{theoremexactseq}) is isogenous 
to a product of components of local Albanese varieties of $\C$.
Indeed Poincare complete reducibility theorem (cf. \cite{lange}) 
implies that 
the 
image of the middle map is isogenous to a direct sum of irreducible 
summands of the last term. 


Denote by the same letter $\phi_j$ the extension of 
a regular map $\phi_j: \PP^2 \setminus \C
\rightarrow \PP^1 \setminus D_j$ to the map $\PP^2 \setminus \S_j \rightarrow 
\PP^1$ where $\S_j$ is the finite collection of indeterminacy points
of the extension of $\phi_j$  to $\PP^2$. Let $\C_{d}=\phi_j^{-1}(d), d \in D_j$.
Then $\C$ contains the union of the closures $\bar \C_d$ of (which are 
possibly reducible and 
non reduced curves).
Each $P \in \S_j$ belongs to at least $Card D_j$ irreducible components and, 
since $Card D_j>1$,  $P$ is a singular point of $\C$.  We claim the following:

\begin{claim}\label{interclaim}
 Resolution $\tilde \PP^2_{\C,P} 
\rightarrow \PP^2$ of the singularity  
at $P$ contains exactly  exceptional curve $E_P$ such that 
the regular extension  $\tilde \phi_j$ of $\phi_j$ 
to $\tilde \PP^2_{\C,P} \rightarrow \PP^1$
induces a finite map 
$\tilde \phi_j: E_P \rightarrow \PP^1$.
\end{claim} 

To see this, consider a sequence 
of blow ups $\tilde \PP^2_{\C,P,h}, h=1,...,N(\C,P)$
of the plane starting with the blow up of 
$\PP^2$ at $P$ and in which the last blow up produces   
the resolution of singularity of $\C$ at $P$.
For each $h$, let $\phi_{j,h}: \tilde \PP^2_{\C,P,h} \rightarrow D_j$
 be the extension of $\phi$ from $\PP^2 \setminus \C$ to
$\tilde \PP^2_{\C,P,h}$. 
For every base point $Q$ of the map $\phi_{j,h}$ on $\tilde \PP^2_{\C,P,h}$
consider the pencil of tangent cones to fibers of the map $\phi_{j,h}$
The fixed (possibly reducible) component 
of the pencil of tangent cones
$T_d, d \in \PP^1$  to curves $\tilde \phi_j^{-1}(d)$
\footnote{i.e. union of lines which are tangent to a component of the curve 
 $\phi_{j,h}^{-1}(d)$ for any $d$} 
either: 
\par a) coincide with the tangent cone $T_d$ to each  curve $\phi_j^{-1}(d)$, or

\par b) there exist $d$ such that the tangent 
cone $T_d$ to $\phi_{j,h}^{-1}(d)$ at $Q$ 
contains a line not belonging to the fixed component
of the pencil of tangent cones.

Since on $\tilde \PP^2_{\C,P}$ (i.e. eventually after sufficiently many blow ups)
no two fibers of $\phi$ intersect,
in a sequence of blow ups 
desingularizing $\C$ at $P$, there is a point $Q$ infinitesimally close to $P$
at which the tangent cones satisfy b). 
At such point $Q \in \tilde \PP^2_{\C,j,h}$ any two distinct fibers of 
$\phi_{j,h}$ admit distinct tangents because otherwise, since 
we have one dimensional linear system, the common tangent 
 to two fibers will belong to the fixed component.
Let $E_P \subset \tilde \PP^2_{\C,j,h+1}$
 be the exceptional curve of the blow up of $\tilde \PP^2_{\C,j,h}$. 
Exceptional curves preceding or following this one on 
the resolution tree (which up to this point did not have vertices
with valency greater than 2!)
 belong to one of the 
fibers of $\phi_j$. Restriction of $\phi_{j,h+1}$ onto $E_P$ 
is the map claimed in (\ref{interclaim}).

Finally, the ramified $\Gamma$-covering of $\PP^2$ 
lifted to $\PP^2_{\C,P}$ and restricted on 
the proper preimage 
of the curve
$E_P$ in $\tilde \PP^2_{\C,P}$ induces the map onto 
$\Gamma_j$-covering of $\PP^1$ ramified at 
$D_j$. Hence the Jacobian of the latter 
covering is a component of the Jacobian of a covering of $E_P$.
It follows from the Corollary \ref{abelianiscyclic} that 
Jacobian of this cover of $E_P$ isogenous to product of 
Jacobians of cyclic covers. Each Jacobian of cyclic 
cover of exceptional curve, in turn, is a component 
of local Albanese variety of singularity with 
appropriately chosen multiplicities of components of 
$\C$ given the by data of the cyclic cover  
of $E_P$ (cf. Theorem \ref{localcontribution}). 
\end{proof}

\bigskip

The following theorem \ref{explicite} 
allows to describe the isogeny class of Albanese varieties 
of abelian covers in explicit examples considered in 
the next section. The Albanese variety of abelian cover
with Galois group $\Gamma$ 
will be obtained as a sum of  
isogeny components of Jacobians 
of abelian covers of the line associated with $\Gamma$ and 
corresponding to the 
positive dimensional components of the characteristic 
variety of $\pi_1(\PP^2\setminus \C)$.
To state the theorem we shall use the following partial order 
on the set of mentioned isogeny components.

\begin{dfn}\label{order} Let $\Psi_i: \B \rightarrow 
\A_i, i \in I$, be a collection of equivariant morphisms of abelian varieties 
endowed with the action of a finite abelian group $\Gamma$.
An isotypical isogeny component of the collection $\A_i$ is 
an abelian variety of the form $S^m$ where $S$ is $\Gamma$-simple
\footnote{i.e. simple in the category of abelian varieties with 
$\Gamma$-action cf.\cite{lange2}}.
Define the partial order of the set of isotypical components  
of $\Pi_{i \in I} \A_i$
as follows: $\A \ge \A'$ 
if and only if 
each $\A$ and $\A'$ belongs to the image of one of  $\Psi_i$ ($i \in I$)
and $\A=S^m,\A'=S^{m'}, m \ge m'$
\end{dfn}

Now we are ready to state 
the following description of the Albanese variety of
abelian cover $\bar X_{\Gamma}$. 

\begin{theorem}\label{explicite} 
Let $\C$ be a plane curve as in Theorem \ref{mainresult}
i.e. with fundamental 
group of the complement satisfying the 
Condition \ref{condition} and all components of characteristic variety 
containing the identity character.
Let $\pi_{\Gamma}: \pi_1(\PP^2 \setminus \C) \rightarrow \Gamma$ 
be a surjection onto an abelian group 
and let $\Gamma_{\Xi_{i}}$ be  corresponding push-out 
group given by diagram (\ref{pushout}).
Let $\bar \PP_{\Gamma_{\Xi_i}}$ denotes the ramified  
cover of $\PP^1$ with covering group $\Gamma_{\Xi_{i}}$
which is the compactification of the cover of the target map 
of $\PP^2\setminus \C \rightarrow \PP^1
\setminus D_i$ corresponding to the component $\Xi_i$.

(1) For any $i$ there are $\Gamma$-equivariant morphisms
\begin{equation}\label{equivmorphisms}
\ \ Alb(\bar X_{\Gamma})  \rightarrow Jac(\bar \PP_{\Gamma_{\Xi_i}})
\end{equation}

(2) Let $A_m, m \in M$ be the set of maximal elements in  
the ordering of isotypical components of collection 
of morphisms in (1). 

Then there is  
an isogeny
\begin{equation}\label{exactsequencenew} 
   Alb(\bar X_{\Gamma}) \rightarrow \oplus_{m \in M} A_m 
\end{equation}

\end{theorem}

\smallskip

\begin{remark} The maps in (\ref{equivmorphisms}) 
corresponding to different characters 
may coincide (this is always the case for example for 
conjugate characters). 
The theorem asserts that selection among jumping characters
and component of characteristic varieties can be made
so that maximal isotypical components in corresponding covers 
provide isotypical decomposition of $Alb(\bar X_{\Gamma})$.
\end{remark}

\begin{proof} Morphisms $\bar X_{\Gamma} \rightarrow \bar \PP^1_{\Gamma_{\Xi_i}}$
were constructed in the beginning of the 
proof of theorem \ref{mainresult}. 

Let $A_m, m \in M$ be collection of maximal 
isotypical components in the Albanese 
varieties which are targets of the maps (\ref{equivmorphisms}).
Composition of a map (\ref{equivmorphisms}) with projection 
on the isogeny components $A_m, m \in M$ gives the 
map $Alb(\bar X_{\Gamma}) \rightarrow A_m$.
Each isogeny component of $Alb(\bar X_{\Gamma})$ is an isogeny component 
in one of varieties $\bar \PP^1_{\Gamma_{\Xi_i}}$ 
and the dimension of $\Gamma$-eigenspace 
corresponding to any character coincides with the dimension 
of $\chi$-eigenspace of the targets (\ref{equivmorphisms}).
Hence the map (\ref{exactsequencenew})
has finite kernel. 

Let $\chi$ be a character having non zero eigenspace
on $H^1(A_m)$.
Then by theorem  \ref{homologycovers} part (2),
$dim H^1(A_m)_{\chi}=dim H^1(\A)_{\chi}=dim H^1(\bar X_{\Gamma})_{\chi}$
where $\A$ is one of the targets of the maps 
(\ref{equivmorphisms}).
Since $H^1(\bar X_{\Gamma})$ is a direct sum
of $\Gamma$-eigenspaces and the image of $H^1(\bar X_{\Gamma})_{\chi}$ is 
non-trivial in exactly one summand in (\ref{exactsequencenew})
one obtains the surjectivity in 
(\ref{exactsequencenew}).
 \end{proof}

\begin{remark} Multiplicities of isotypical components
$A_m$ are poorly understood in general as well as jumping characters
(cf. \cite{jose} where the problem of bounding the 
multiplicities of the roots of Alexander polynomials of 
the complements to plane curves, which are in correspondence with the  
jumping characters, is discussed). 
Nevertheless in all known examples, the above theorem 
is sufficient to completely determine isogeny class of Albanese varieties 
of abelian covers. 
\end{remark}

\section{Albanese varieties of abelian covers ramified over arrangements of lines.}\label{examplessection}

\bigskip

In the case when ramification set is an arrangement of lines 
theorems \ref{mainresult} and \ref{explicite}  
yield considerably simpler than in general case  results. 
We shall start with: 
\smallskip
\begin{corollary}\label{corexplicite} 
Let $\A$ be an arrangement of lines in $\PP^2$ with double and 
triple points only which satisfies the assumptions\footnote{i.e. we  consider only the cases when 
all irreducible 
components of characteristic variety contain the identity and 
also Condition \ref{condition} is satisfied.} of Theorem \ref{mainresult}.
Let $X_n(\A)$ be a compactification of the abelian 
cover of $\PP^2 \setminus \A$ corresponding to the surjection
$H_1(\PP^2 \setminus \A,\ZZ) \rightarrow H_1(\PP^2 \setminus \A,\ZZ/n\ZZ)$.

(1) Albanese variety of $X_n(\A)$ is isogenous to a product of 
isogeny components of Jacobians of Fermat curves.

(2) $Alb(X_n(\A))$ is isogenous to a product 
and of Jacobians of Fermat curves if 

(a) none of the characters in $\Char H^1(\PP^2\setminus \A,\ZZ/n\ZZ)
\subset \Char H_1(\PP^2\setminus \A,\ZZ)$ is a jumping 
character in the characteristic variety of $\pi_1(\PP^2\setminus \A)$
and 

(b) the pencils corresponding 
to positive dimensional components have no multiple fibers. 

\end{corollary}

\begin{proof}
 Each component of characteristic variety having a positive dimension 
corresponds to the map $\PP^2 \setminus \A \rightarrow \PP^1 \setminus D$ 
where
${\rm Card} D=3$. Those induce maps of $Alb(X_n(\A))$ onto
 the Jacobians of abelian covers of $\PP^1$
ramified along corresponding $D$. The Jacobian of such
abelian cover of $\PP^1$  
is a component of the Jacobian of Fermat curve.
(cf. Corollary \ref{abelianiscyclic} with $k=2$). 
Hence the maximal isotypical isogeny components 
(cf. Theorem \ref{explicite})
are components of Jacobians of Fermat curves and  
therefore part (1)
follows from theorem \ref{explicite} i.e.
$Alb(X_n(\A))$ is isogenous to a product of 
components of Fermat curves.
Note that the Theorem \ref{mainresult} for arrangements 
of lines with double and triple points can be obtained 
follows from these arguments.
Indeed, the isogeny components of Jacobians 
of Fermat curves are Jacobians 
of cyclic covers of $\PP^1$ ramified at three points
(cf. \cite{gross},\cite{coleman}) and 
Jacobians of cyclic covers of $\PP^1$
ramified at three points are local Albanese varieties 
of non-reduced singularities of the form $x^{a_1}(x-y)^{a_2}y^{a_3}$
(cf. Example \ref{examplenonreduced}).

If characteristic variety does not have jumping characters
in subgroup $\Char H_1(\PP^2 \setminus \A,\ZZ/n\ZZ)$ of 
$\Char \pi_1(\PP^1\setminus \A)$
then 
$Alb(X_n(\A))$ is just a product of Jacobians corresponding 
to positive dimensional 
components of characteristic variety (i.e. 
there are no ``corrections'' 
in $A_m$  coming from Jacobians of covers corresponding to jumping 
characters). The assumption about absence of multiple fibers 
implies that map of $X_n(\A)$
corresponding to each positive dimensional 
component of characteristic variety of $\A$ has as target the 
cover as in Remark \ref{fermat} i.e. a Fermat curve.
Hence $Alb(X_n(\A))$ is a product 
of Jacobians of Fermat curve and we obtain part of (2). 
\end{proof}
\smallskip

\begin{example}\label{ceva} Consider Ceva arrangement $xyz(x-z)(y-z)(x-y)=0$
and the universal $\ZZ_5$ cover (with the covering group which is the 
quotient of $\ZZ_5^6$ by the cyclic subgroup generated by $(1,1,1,1,1,1)$.
Then the irregularity of the corresponding abelian cover is $30$ (cf. \cite{ishida}, \cite{charvar}
section 3.3 ex.2). The characteristic variety consists of five 
2-dimensional components $\Xi_i, i=1,...,5$ 
(cf. \cite{charvar}), each being the pull back
of $H^1(\PP^1\setminus D,\CC^*), {\rm Card} D=3$ via 
either a linear projection from one of 4 triple points 
or via a pencil of quadrics three degenerate fiber of which 
form the 6 lines of the arrangement.
Each of these 5 pencils 
induces a map on the abelian cover of $\PP^1$ branched at 3 points, 
which has as the Galois group the quotient of $\oplus_1^3\mu_5$ by 
the diagonally embedded group of roots of unity $\mu_5$ of degree $5$.
This cover, i.e. $\bar \PP_{\Xi_i}, i=1,..,5$,
 is the Fermat curve of degree $5$. The  Jacobian of degree-5 Fermat curve
is isogenous to 
a product of Jacobians of three curves $C_i, i=1,2,3$ of genus 2 
each one being a cyclic cover of $\PP^1$ ramified 
at three points.
(cf. \cite{coleman},\cite{koblitz}). 
Hence the Albanese variety of this abelian cover is isogenous to a product 
of 15 copies of the Jacobian of ramified at three points cover 
of $\PP^1$ 
of degree 5. 
In this example there are no jumping characters 
(in particular in $\Char H_1(\PP^2 \setminus \A,\ZZ/5\ZZ)$)  
and the isogeny can be derived from Corollary \ref{corexplicite} 
\end{example}

\begin{example} Consider again Ceva arrangement and calculate the abelian 
component of (semi-abelian) Albanese variety (cf. section \ref{semiabelian})
of its Milnor fiber $ M$ given by 
$w^6=\Pi l_i$. Notice 
that the characteristic polynomial of the monodromy is $(t-1)^5(t^2+t+1)$
(cf. \cite{charvar}). 
The $\zeta_3$-eigenspace of $H^1(M,\CC)$ can be identified 
with the contribution in sum (\ref{sakuma}) of the pullback  
of the character $\chi$
of $\PP^1\setminus D$
via the pencil of quadrics formed by lines of the arrangement. Here $D$ is the triple 
of points corresponding to the reducible quadrics in the pencil and 
$\chi$ is the character taking the same value $\omega_3$ on standard
generators if $\pi_1(\PP^1 \setminus D)$. 
This pencil can be lifted to the elliptic  pencil on a compactification 
of $M$ onto 3-fold cyclic cover of $\PP^1$ ramified at $D$ and
corresponding to $\Ker \chi$. Moreover,  above expression for the 
characteristic polynomial of the monodromy shows that the map 
induced by this pencil is isogeny i.e. the abelian (i.e. compact) 
component of 
the Albanese of $M$ is the elliptic curve $E_0$.
The semi-abelian variety with is the Albanese variety 
of $M$ is an extension:
\begin{equation}
   0 \rightarrow (\CC^*)^5\rightarrow Alb(M) \rightarrow E_0
\rightarrow 0 
\end{equation}
\end{example}
\smallskip
\begin{example}\label{dualflexes} Consider abelian cover of $\PP^2$ ramified along arrangement 
of lines dual to 9 inflection points of a smooth cubic 
with Galois group $\ZZ_n^9/\ZZ_n$. 
This arrangement has 9 lines and 12 triple points. An explicit equation 
is as follows:
\begin{equation}
    (x^3-y^3)(y^3-z^3)(z^3-x^3)=0
\end{equation}
The characteristic variety 
consists of 12 components corresponding to 12 triple points and 4 additional 
two-dimensional components intersecting along cyclic subgroup of order 3.
Characters at the intersection are jumping and have depth $2$ 
(cf. \cite{dimcapencils},\cite{naie})
while depth of generic character in each positive dimensional 
component is 1. In coordinates of $\Char \pi_1(\PP^2\setminus \A)$ 
corresponding to components of $\A$ described jumping
characters have the form $(\omega,....,\omega), \omega^3=1$.
 
In the case $n=5$, in which according to Hirzebruch 
one obtains a quotient of the unit ball, the Albanese variety is isogenous to 
the product of 16 copies of Fermat curve of degree 5, 
as follows from Corollary \ref{corexplicite} (2) or equivalently 
48 copies of curves of Jacobians of curves of genus 2 with 
automorphism of order 10 or, what is the the same, the 2-dimensional 
 variety of CM type 
corresponding to cyclotomic field $\QQ(\zeta_5)$.
For arbitrary $n$ such that ${\rm gcd}(3,n)=1$
one get several copies of Jacobians of Fermat curves of degree $n$ 
corresponding to 
components of characteristic variety.

If $n$ is divisible by $3$, i.e. the jumping characters are present,
then the condition \ref{condition}
should be verified. 
To this end, we shall reinterpret the part of this condition 
dealing with the map between 
the cohomology of local systems.
The cohomology of the local systems appearing in (\ref{condition})
can be identified with the eigenspaces of the (co)homology of 
abelian covers (cf.\cite{charvar}). More precisely, 
the $\chi$-eigenspace can be identified with the cohomology of the local 
system corresponding to the character $\chi$.
The eigenspace   
corresponding to the character belonging to 
4 irreducible components of characteristic variety in turn 
can be interpreted as the dual space of $H^1(\PP^2,\I_Z(3))$
where $Z\subset \PP^2$ is the subscheme of triple points
(cf. \cite[(3.2.14),(3.2.15)]{charvar} and corresponding remark). 
On the other hand, each of the above 4 components 
corresponds to a selection of a subset $Z_i \subset Z, 
Card Z_i=9$, cf. \cite[Section 3.3,Example 3]{charvar}
for description of these subsets, each of which is 
a complete intersection of two cubic curves.
The cohomology of generic local system in such component 
is identified with the dual space of $H^1(\PP^2,\I_{Z_i}(3))$.
The condition \ref{condition} is interpreted as
injectivity of the map 
\begin{equation}\label{surjection}
H^1(\PP^2,\I_Z(3)) 
\rightarrow \bigoplus_{i=1}^{i=4} H^1(\PP^2,\I_{Z_i}(3))
\end{equation} 
induced by injections $\I_Z \rightarrow \I_{Z_i}, i=1,...4$.
This injectivity is readily seen e.g. by interpreting
terms in (\ref{surjection}) using 
standard sequence: $0 \rightarrow \I_Z \rightarrow \O_{\PP^2} 
\rightarrow \O_Z \rightarrow 0$ and similar sequences for $Z_i$. 

Implication of verification of Condition \ref{condition}
is that in this case the product of Jacobians of 
Fermat curves which are the Jacobians corresponding 
to positive dimensional components of characteristic variety
must be factored by the product $E_0^{\kappa-\delta}$
where $\kappa$ is the number of components containing
a jumping character (taking value 
$exp({{2 \sqrt{-1} \pi}\over 3})$
or  $exp({{4 \sqrt{-1} \pi}\over 3})$ on all 9 lines of arrangement)
and $\delta$ is the depth of the jumping character
\footnote{cf.\cite{naie}, Prop. 4.8.
This effect of characters in the intersection 
of several components of characteristic varieties is 
erroneously omitted in the final formula in Example 3 in 
section 3.3 of \cite{charvar}.}.

In the case $n=3$ the abelian cover with the
covering group $\ZZ_3^9/\ZZ_3$ one obtains from theorem \ref{explicite}
or Corollary \ref{explicite} 
\begin{equation}
     Alb(\bar \PP^2_{\ZZ_3^8})=E_0^{16}/E_0^2=E_0^{14}
\end{equation}
\end{example}
Indeed, in this case $\kappa=4, \delta=2$.

In the case $3 \vert n, n>3$, the product of Jacobians 
corresponding to positive dimensional components has several 
copies of $E_0$ as isogeny components and $Alb(X_n)$ is the 
quotient of this product by $E_0^{\kappa-\delta}=E_0^2$.

\begin{example} Consider Hesse arrangement $\mathcal H$ formed by 12 lines 
containing 
9 inflection points of a smooth cubic. It was shown in \cite{charvar} (cf. 
section 3, example 5) that the characteristic variety of the fundamental group 
of the complement to this arrangement consists of 10 three-dimensional 
components and 54 two-dimensional components none 
of which belongs to a three-dimensional component
(intersection of components must be zero dimensional).
As earlier, it is convenient to describe  components in terms of corresponding 
pencils i.e. maps $\PP^2 \setminus \H \rightarrow \PP^1 \setminus h$ where 
$h$ is a set of points of cardinality $4$ or $3$ so that the characters
in each component formed by pullbacks via these maps. The pencils corresponding 
to components of dimension 3 are linear projections from each of 9 quadruple 
points and the additional pencil is the pencil of curves of degree 3 containing 
4 cubic curves each being a union 
of a triple of lines in the arrangement $\cal H $.
The 54 maps $\PP^2 \setminus \H \rightarrow \PP^1 \setminus h$ ($Card \ h=3$)
are restrictions of the maps corresponding to the pencil of quadrics 
union of which are 6-tuples of lines in $\H$ forming a Ceva 
arrangement.\footnote{This was explained in \cite{charvar}. Recall that in interpretation 
of inflection points of the cubic as points in affine plane over field $\FF_3$, 
the twelve lines correspond to the full set of lines in this plane and 6 tuples 
are in one to one correspondence with quadruples of points in this finite plane 
no three of which are collinear. Counting first ordered quadruples of this type 
one sees that there are $9 \times 8$ choices for the first 
two points, $6$ choices for the third point 
(it cannot be the third point on the 
line containing first two) and $3$ choices for the forth). 
Therefore one get 54 unordered 
quadruples of points and hence 54 6-tuples of lines.}

The pencil corresponding to 3-dimensional component of characteristic variety 
induces the map of abelian cover of the plane ramified along $\H$ 
with Galois group $(\ZZ_3)^{12}/\ZZ_3$ on the maximal abelian cover $\ZZ_3$ cover 
 of $\PP^1$ ramified at 4 points.
In particular the Albanese variety in question maps 
onto the Jacobian $J_{10}$ of curve of genus 10. 
Similarly each 2-dimensional component of characteristic variety induces map 
of Albanese of abelian cover of $\PP^2$ onto maximal abelian 3-cover of 
$\PP^1$ ramified at 3 points. The latter is Fermat curve of degree i.e. 
the elliptic curve with $j$-invariant zero.

We obtain that the Albanese variety of the cover considered by Hirzebruch (cf.\cite{hirz})
is isogenous to 

\begin{equation}
   J_{10}^{10} \times E_0^{54}
\end{equation}
\end{example}

\begin{example}{Variety of lines on a Fermat hypersurface}
Previous results imply immediately the following:
\begin{theorem} Let $F_3$ be variety if lines on Fermat 
cubic threefold:
\begin{equation}
x_0^3+x_1^3+x_2^3+x_3^3+x_4^3=0  
\end{equation}
Then there is an isogeny:
\begin{equation}\label{fanoisogeny}
   Alb(F_3)=E_0^5
\end{equation}
\end{theorem}
This isogeny was observed recently \cite{toledo}. 
Also, Roulleau cf. \cite{roulleau} obtained the isomorphism class 
of the Albanese variety of Fermat cubic threefold.

\begin{proof} It follows from discussion in \cite{terasoma}
that Fano surface $F_3$ is abelian cover of degree $3^4$ of 
$\PP^2$ ramified over Ceva arrangement. Hence the isogeny 
(\ref{fanoisogeny}) follows as in example \ref{ceva}.
\end{proof}
\end{example}

\section{Applications}

\subsection{Mordell-Weil ranks of isotrivial families of
abelian varieties}

Recall the following (cf. \cite{annalen})
\begin{prop}\label{isotrivial}
 Let $\A \rightarrow \PP^2$ be a regular model 
of an isotrivial abelian 
variety over $\CC(x,y)$ with a smooth fiber $A$. 
Assume that there is a ramified {\it abelian} 
cover $X \rightarrow \PP^2$ such that the pullback of $\A$ to $X$ 
is trivial abelian variety over $X$. 
Let $\Gamma$ be the Galois group of $\CC(X)/\CC(x,y)$. 
Then the trivialization of $\A$ over $X$ yields the action of $\Gamma$ on 
$A$ and the Mordell-Weil rank of $\A$ is equal 
to $\dim_{\QQ} Hom_{\Gamma} (Alb(X),\A) \otimes \QQ$.
\end{prop}

Let $A$ be an abelian variety over $\CC$.
Given an abelian cover 
$X \rightarrow \PP^2$ with covering group $\Gamma$
and a homomorphism $\Gamma \rightarrow Aut A$, 
an example of isotrivial abelian variety 
over $\CC(x,y)$ as in Prop.\ref{isotrivial}
can be obtained as a resolution of singularities of 
\begin{equation}\label{isotrivialformula}
  \A_X=X \times A/\Gamma
\end{equation} 
where $\Gamma$ acts on $X \times A$ diagonally: 
$(x,a) \rightarrow (\gamma \cdot x,\gamma \cdot a), 
\gamma \in \Gamma, x \in X, a \in A)$. The map 
$\A_X \rightarrow X/\Gamma=\PP^2$ gives to $\A_X$ 
a structure of isotrivial abelian variety over $\CC(x,y)$.

Calculations of Albanese varieties in examples of previous sections 
yield values of Mordell-Weil ranks of isotrivial  
abelian varieties in many examples as in 
Prop. \ref{isotrivial}.  
\begin{example}Let $J_{2,5}$ denote the Jacobian of a smooth projective model 
of genus $2$ curve $C$ given by equation: $y^5=x^2(x-1)^2$ 
(i.e. one of the curves $C_i$ in Example \ref{ceva}).
Assume that the direct sum $\Gamma=\ZZ_5^5$ acts 
on $C$ so that the generator of 
each summand acts as the multiplication by $\zeta, \zeta=exp({{2 \pi i}
\over 5}):
(x,y) \rightarrow (x,\zeta y)$ 
(cf. \ref{ceva}). This induces the action of $\ZZ_5^5$ on $J_{2,5}=Jac(C)$.
In example \ref{ceva}, we viewed $\Gamma$ as the quotient of $\ZZ_5^6$ 
by $(1,1,1,1,1,1)$, so that each summand
corresponds to monodromy about one of 6 lines in Ceva arrangement. 
Then an identification of $\ZZ_5^5$ and $\ZZ_5^6/\ZZ_5$ can be obtained 
by identifying 
the former group with the image in the latter of the  
subgroup of $\ZZ_5^6$ of elements $(a_1,a_2,a_3,a_4,a_5,
-\sum_{i=1}^{i=5}) a_i, a_i \in \ZZ_5$.
In such presentation of $\Gamma$, the action of first 5 components 
of elements in $\ZZ_5^6$ on $C$ is given by multiplication 
by $\zeta$ while action of the  last component on $C$ is trivial. 

Consider isotrivial family $\A_X$ of 
abelian varieties over $\PP^2$ given by (\ref{isotrivialformula}) with the zero set of discriminant 
being the Ceva arrangement of lines which is the quotient of 
$X \times J_{2,5}$, where $X$ is the abelian cover with the covering 
group $\ZZ_5^5$ considered in example \ref{ceva}. 
The action of $\Gamma$ is the 
diagonal action of $\Gamma=\ZZ_5^5$ as in (\ref{isotrivialformula}). 
The Albanese variety of the abelian cover $X$ in example \ref{ceva}
is isogenous to $(J_{2,5})^{15}$ (cf. (\ref{ceva}))
and hence the 
rank of the Mordell-Weil group of the quotient is equal
to 
\begin{equation}\label{hom}
\rk Hom_{\ZZ_5^5} (J_{2,5}^{15},J_{2,5}) \otimes \QQ
\end{equation}

The characters of representation of $\Gamma=\ZZ_5^6/\ZZ_5$ 
on $H_1(J_{2,3}^{15})$ are the characters of representation 
of $\Gamma$ on $H^{1,0}(X,\CC)$ i.e. the characters from 
the characteristic variety of Ceva arrangement.
Clearly neither of two characters for described
above action of $\Gamma$ on $H^1(C,\CC)$, having the 
form $(a,a,a,a,a,1), a \in \ZZ_5$ in the basis of $\Char \Gamma$
dual to the one coming from direct sum presentation of $\ZZ_5^6$, 
belongs to the characteristic variety of Ceva arrangement.
Hence the rank (\ref{hom}) is zero.

\end{example}

\subsection{Periodicity of Albanese varieties}

\begin{theorem}\label{periodicity}
Let $\C$ be a curve in $\PP^2$ such that there exist a 
surjection $\pi: \pi_1(\PP^2\setminus \C)
\rightarrow \ZZ$ \footnote{For any curve in $\C$ (including 
irreducible in which case $H_1(\PP^1\setminus \CC,\ZZ)=\ZZ/({\rm deg}\C)\ZZ$)
adding to $\C$ a generic line in $\PP^2$ yields a curve admitting 
such surjection cf. \cite{duke}.}. Consider two sequences
of cyclic covers composed of ramified and unramified covers  
corresponding to surjections $\pi_n: \pi_1(\PP^2\setminus \C)
\rightarrow \ZZ \rightarrow \ZZ/n\ZZ$  

(1) The sequence of isogeny classes of Albanese varieties of 
a tower of cyclic branched covers with given ramification 
locus $\C$ corresponding to surjections $\pi_n$ is periodic.\footnote{i.e. exist $N \in \NN$ such that Albanese varieties 
of cyclic covers corresponding to $\pi_n,\pi_{n'}$ with 
$n\equiv n'\ \mod N$ are isogeneous.}

(2) The sequence of isogeny classes of semi-abelian varieties
which are Albanese varieties of unbranched covers 
a complement to a curve $\C$ corresponding to surjections
$\pi_n$ is periodic.

\end{theorem} 

\begin{proof} Let 
$\Delta_{\pi}(t)$
be the Alexander polynomial of $\C$ corresponding to the surjection 
$\pi$ (cf. \cite{duke}). For each root $\xi$ of $\Delta_{\pi}(t)$  
let $n_{\xi}$ be its order (recall that any root 
of  Alexander polynomial of an algebraic curve   
is a root of unity). 
For each set $\Xi$ of distinct roots of $\Delta_{\pi}(t)$ let 
$n_{\Xi}=lcm(n_{\xi}), \xi \in \Xi$ and let $N$ be the least common multiple 
of integers $n_{\Xi}$. 
To each congruence class modulo $N$ 
corresponds a subset $\Xi$ (possibly empty) 
such that integers in this class are 
divisible by exactly one (or none) among the integers $n_{\Xi}$. 

The rank of $H_n(X_{n})$ depends only on the number of roots 
$\xi$ such that $\xi^n=1$ (cf. \ref{homologycovers})
i.e. on $n \mod N$.
More precisely, let $X_n$ (resp.
$\bar X_n$) denotes unramified (resp. ramified) cover of $\PP^2\setminus \C$ 
(resp. $\PP^2$). Then  
$H_1(X_n,\CC) \rightarrow H_1(X_{n_N},\CC)$
(resp. $H_1(\bar X_n,\CC) \rightarrow H_1(\bar X_{n_N},\CC)$)
are isomorphisms for all $n$ belonging to one 
of the congruence class modulo 
$N$. For $n$ not belonging to any of these 
congruence classes, one has $H_1(X_n,\CC)=H_1(\bar X_n,\CC)=0$. 
Moreover 
the map $H_1(X_n,\ZZ) \rightarrow H_1(X_{n_N},\ZZ)$
(resp. $H_1(\bar X_n,\ZZ) \rightarrow H_1(\bar X_{n_N},\ZZ)$)
is injective (resp. has finite kernel and co-kernel). Hence
the isogeny class of Albanese variety of $X_n$ with $n$ in one and 
only one  congruence class as above is constant.  Hence the claims (1) and (2) follow.
\end{proof}

\begin{remark} This result can be compared with results on 
periodicity properties of Betti and Hodge numbers.
For a curve $\C$ in $\CC^2$ which a union of $r$ components,
let $h^{1,0}_a(n)$ (resp. $h^{1,0}_c(n)$) denote the sequence of 
the Hodge numbers of a smooth compactification $X_a(n)$ (resp. $X_c(n)$) of
abelian (resp. cyclic) covers of the complement
in the tower of abelian (resp. cyclic) cover of $\PP^2$
corresponding to surjection 
$\pi_1(\CC^2 \setminus \C) \rightarrow \ZZ^r \rightarrow (\ZZ/n\ZZ)^r$ 
(resp. surjection onto $\ZZ_n$  which is the composition of the latter 
surjection with summing up of the coordinates in $\ZZ^r/n\ZZ^r$).

It follows from \cite{eko} that the 
sequence $h^{1,0}_a(n)$ is polynomial periodic 
(similarly, \cite{duke} implies that sequence $h^{1,0}_c(n)$ 
is periodic).
Recall that $n \rightarrow a(n) \in \NN$ is polynomial periodic if there 
are periodic functions $a_i(n)$ such that $a(n)=\sum a_i(n)n^i$.

Now let $K_{\A\B}$ be the K-group of motives of abelian varieties over $\CC$
up to isogeny. More precisely this is the $K$-group of the category $\A\B$
defined as follows. 
The objects of $\A\B$ are abelian varieties over $\CC$ and 
$Hom_{\A\B}(A,A')$ is the group of homomorphism between $A,A'$ tensored 
with $\QQ$.
The $K$-group $K_{\A\B}$  
is an (infinitely generated) $\ZZ$-module with canonical 
surjection $dim: K_{\A\B} \rightarrow \ZZ$ given by $A\mapsto dim A$.
The theorem \ref{periodicity} 
implies that the sequence $Alb(X_c(n)) \in K_{\A\B}$ is 
periodic.
However isogeny components of Albanese varieties of 
{\it abelian} possibly non-cyclic covers span an infinitely generated subgroup 
of the K-group of isogeny classes of abelian varieties. 

In particular, there are no periodic functions $a_i(n) \in K_{\A\B}$ such that 
$Alb(X_a(n))=a_i(n)n^i$ (though as was mentioned $dim Alb(X_a(n))$ is 
polynomially periodic). Details of this will be presented elsewhere.

\end{remark}

\end{document}